\documentclass[12pt, reqno]{amsart}
\usepackage{amsmath, amsthm, amscd, amsfonts, amssymb, graphicx, color}
\usepackage[bookmarksnumbered, colorlinks, plainpages]{hyperref}
\hypersetup{colorlinks=true,linkcolor=red, anchorcolor=green,
citecolor=cyan, urlcolor=red, filecolor=magenta, pdftoolbar=true}

\newtheorem{theorem}{Theorem}[section]
\newtheorem{lemma}[theorem]{Lemma}
\newtheorem{proposition}[theorem]{Proposition}
\newtheorem{corollary}[theorem]{Corollary}
\theoremstyle{definition}
\newtheorem{definition}[theorem]{Definition}

\theoremstyle{remark}
\newtheorem{remark}[theorem]{Remark}
\newtheorem{remarks}[theorem]{Remarks}
\numberwithin{equation}{section}
\begin{document}

\setcounter{page}{1}

\title{A note on weak positive matrices, finite mass measures  and hyponormal weighted shifts}

\author[H. El-Azhar, K. Idrissi, E.H. Zerouali]{H. El-Azhar $^{1*}$, K. Idrissi $^1$, \MakeLowercase{and} E.H. Zerouali $^{1}$}

\address{$^{1}$Center of mathematical research of Rabat, P.O. box 1014, Department of Mathematics, Faculty of
sciences, Mohammed V university in Rabat, Rabat, Morocco.}
\email{\textcolor[rgb]{0.00,0.00,0.84}{elazharhamza@gmail.com;
kaissar.idrissi@gmail.com; zerouali@fsr.ac.ma}}



\begin{abstract}
We study the class of Hankel matrices for which the $k\times
k$-block-matrices are positive semi-definite. We prove that a
$k\times k$-block-matrix has non zero determinant if and only if all
$k\times k$-block matrices have non zero determinant. We use this
result to extend  the notion of propagation phenomena to
$k$-hyponormal weighted shifts. Finally we give a study on
invariance of $k$-hyponormal weighted shifts  under one rank
perturbation.
\end{abstract}
\maketitle
\section{Introduction}

Let $\mathcal{H}$ be a complex Hilbert space and let
$\mathcal{L(H)}$ be the algebra of bounded operators on
$\mathcal{H}$. We denote by $[T,S]:= TS-ST$ the commutator of $S$
and $T$ in $\mathcal{L(H)}$. An operator $T\in\mathcal{L(H)}$ is
said to be normal if $[T^*,T]=0$, to be hyponormal if $[T^*,T]\geq
0$ and to be subnormal if $T=N_{|\mathcal{H}}$, where $N$ is a
normal operator on some Hilbert space $\mathcal{K} \supseteq
\mathcal{H}$.

The concepts of subnormal and hyponormal operators were introduced
by Paul R. Halmos  in \cite{Hal}. The first notion, hyponormal,
reflects the geometric nature of normality  with the corresponding
implications in terms of  positive matrices; while subnormal is
intimately related to the notion of analyticity for complex
functions, through the restriction of the functional calculus to
invariant subspaces.

In order to establish a bridge between operator theory and matrix
theory, we recall the Bram-Halmos criterion for subnormality
\cite{Bra, Con}, which says that an operator $T$ is subnormal if and
only if
  \begin{equation}\label{1.1}
  \sum_{i,j\leq k} \langle T^ix_j,T^jx_i\rangle \geq 0 \qquad k\geq
  0,
  \end{equation}
for any $x_0, x_1, \ldots,x_k \in \mathcal{H}$. An application of
the Choleski algorithm for operator matrices shows that \eqref{1.1}
is equivalent to the positivity test

  \begin{equation}\label{1.2}
  M_k(T):=\left( \begin{matrix}
 [T^*,T]     &[T^{*2},T]      &\ldots &[T^{*k},T]\\
 [T^*,T^2]  &[T^{*2},T^2]    &\ldots &[T^{*k},T^2]\\
 \vdots      &\vdots           &\ddots &\vdots\\
 [T^*,T^k]  &[T^{*2},T^k]    &\ldots &[T^{*k},T^k]
 \end{matrix} \right) \geq 0 \qquad  k\geq0.
  \end{equation}

To illustrate and to study the gap between subnormal and hyponormal
operators A. Athavale \cite{Ath} introduced the classes of
$k$-hyponormal operators as follows, an operator
$T\in\mathcal{L(H)}$ is  $k$-hyponormal if $M_k(T)\geq0$. Clearly
$$ T \mbox{   is subnormal } \Leftrightarrow   T    \mbox{ is } k-\mbox{hyponormal for every }
 k\in\mathbb{N}.$$  Moreover, 
$$
(k+1)\textrm{-hyponormal} \Rightarrow k\textrm{-hyponormal}
\Rightarrow 1\textrm{-hyponormal} \Leftrightarrow
\textrm{hyponormal}.
$$
Weighted shifts, defined below,   provide several  examples and counter examples in
operator theory and hence are an important motivation in the
analysis  of operators.  Given a bounded sequence of positive
numbers $\alpha\equiv\{\alpha_n\}_{n\ge 0}$ (called weights), the
unilateral weighted shift $W_\alpha$ associated with $\alpha$ is the
bounded operator on $\ell^2(\mathbb{N})$ defined by $W_\alpha
e_n:=\alpha_n e_{n+1}$ for every $n\geq0$, where
$\{e_n\}_{n=0}^\infty$ is the canonical orthonormal basis for
$\ell^2$; the moments of $\alpha$ are defined by $\gamma_0:=1$,
$\gamma_{n+1}:=\alpha_n^2\gamma_n (n\geq 0)$. It is straightforward
to check that $W_\alpha$ can never be normal, and that $W_\alpha$ is
hyponormal if and only if $\alpha_n\leq \alpha_{n+1}$ for all $n\geq
0$.

The Stieltjes moment problem associated with a given sequence
$\{\gamma_{n}\}_{n\ge 0}$ entails finding a positive Borel measure
$\mu$ supported in  $\mathbb{R_+}$ such that
 \begin{equation}\label{i.0}
\gamma_{n}=\int_{\mathbb{R}_+} {t}^nd\mu\: \: \:  \mbox{ for every }
n \ge 0.
\end{equation}

When the moment problem owns a solution $\mu$, then $\mu$  is said
to be a representing measure of the moment sequence $\gamma_n$.  The
well known Berger theorem says that a weighted shift $W_\alpha$ is
subnormal precisely when the sequence of its moments is a moment
sequence of a positive measure supported in $[0,\|W_\alpha\|]$.

 A description of subnormality for an abstract operator $T$ in
terms of weighted shifts can be found in \cite{Lam}. Namely, a one
to one operator $T$ is  subnormal if and only if, for each $h\neq 0$
in $\mathcal{H}$, the weighted shift associated  with the weight
sequence $\{\| T^{n+1}h \| / \| T^{n}h \| \}$ is subnormal.

J. Stampfli in \cite{Sta} (see also \cite{BRZ}) showed that for
subnormal weighted shifts $W_\alpha$, a propagation phenomenon
occurs which forces the flatness of $W_\alpha$ whenever two equal
weights are present. That is, if $\alpha_k=\alpha_{k+1}$  for some
$k\geq0$, then $\alpha_n=\alpha_{n+1}$  for every  $n\geq1$. Later,
R. Curto proved, in \cite{Cur}, that the above result remains valid
for 2-hyponormal weighted shifts. Our main goal in this note is to
generalize the propagation phenomena in order to study the gap
between different classes of $k$-hyponormal weighted shifts. To this
aim we introduce the notion of $k$-positive matrices (or sequences),
and we investigate this concept to exhibit some useful results.

This paper is organized as follows. We define in Section 2 the concept of  
$k$-positive matrices, and we give some of their  basic properties.
Section 3 is devoted to the statement of  the main results. In
Section 4 we give an elementary proof for a result due to
Curto-Fialkow. That is, a $k$-positive matrix has a $k\times
k$-sub-matrix with zero determinant if and only if all $k\times
k$-sub-matrices have zero determinant. We devote Section 5 to a
characterization of finite combinations of Dirac measures on
$\mathbb{R}_+$ in terms of moment. In the last section we study the
invariance of $k$-hyponormal weighted shifts  under one rank
perturbation, and a simple algorithm to calculate the stable
perturbation intervals.

\section{$k$-positive Hankel matrices}
Given a sequence of non negative numbers $\gamma= \{\gamma_n\}_{n\ge
0}$, the associated Hankel matrix is built as follows
  \begin{equation*}
M_\gamma :=(\gamma_{i+j})_{i,j}=\left( \begin{matrix}
\gamma_0     &\gamma_1  &\gamma_2    &\gamma_3 \\
\gamma_1  &\gamma_2 &\gamma_3   &\ldots\\
\gamma_2 & \gamma_3 &\dots &\dots \\
\gamma_3 & \vdots  &\ddots &\ddots
\end{matrix} \right).
\end{equation*}

For $k,n\in\mathbb{N}$,  we denote by $[M_\gamma]^n_k$ the Hankel
$(k+1)\times (k+1)$-sub-matrix

 \begin{equation}\label{1}
[M_\gamma]^n_k=\left( \begin{matrix}
\gamma_n     &\gamma_{n+1}      &\ldots     &\gamma_{n+k} \\
\gamma_{n+1}  &\gamma_{n+2}    &\ldots      &\gamma_{n+k+1}\\
\vdots      &\ddots           &\ddots       &\vdots\\
\gamma_{n+k}        &\gamma_{n+k+1} &\ldots &\gamma_{n+2k}
\end{matrix} \right).
\end{equation}

 \begin{definition}
A Hankel matrix $M_\gamma$, or a sequence $\gamma$, is said to be
$k$-positive if for every $n\in \mathbb{N}$, the  $(k+1)\times
(k+1)$-sub-matrix  $[M_\gamma]_k^n$ is  positive semi-definite.
\end{definition}

Clearly  $M_\gamma$  is $k$-positive means that $\langle M_\gamma
x,x\rangle\geq 0$ for every  $x=\sum\limits_{i=0}^kx_{i}e_{n+i}$ and
$n\in \mathbb{N}$. Also, it is easy to see that the set of
$k$-positive matrices, denoted by $\mathcal{C}_+^k$, is a convex
cone, and that $\mathcal{C}_+^{k+1} \subset \mathcal{C}_+^{k},\text{
for all } k\in \mathbb{N}$.

 Further immediate properties and examples are given in the next
 remark:
\begin{remark}\label{rq}\qquad

    \begin{enumerate}
        \item $\mathcal{C}_+^0 $ is the set of all matrices with non negative entries.
        \item $\gamma=\{\gamma_n\}_{n\ge 0}$ is $1$-positive if and only if $\gamma$
        is non-negative and log-convex.
        \item \label{3} Let $\mu$ be a positive finite measure such that $\textrm{Supp}(\mu)\subset \mathbb{R}_+$ and $\mathbb{R}[X]\subset L^1(\mathbb{R}_+,\mu)$.
         By Stieltjes's Theorem \cite[page 76]{akh}  for every  $k\in\mathbb{N}$,
        the Hankel matrices $(\gamma_{i+j})_{0\leq i,j\leq k}$ and $(\gamma_{i+j+1})_{0\leq i,j\leq k}$
        are positives semi-definite. In other words  $\gamma$ is $k$-positive for all $k\in
     \mathbb{N}$.
    \end{enumerate}
\end{remark}

\section{main results}
From the log-convexity of $k$-positive sequences, we deduce the next
useful remark
   \begin{proposition}\label{1 positive}
    Let $M\in \mathcal{C}_+^k$.  If $\gamma_{n_0}=0$, for some $n_0\in \mathbb{N}$, then  $\gamma_{n}=0$ for every $n \ge 1$.
\end{proposition}
Proposition \ref{1 positive} states  that a propagation phenomena
occurs,  in the sense that if  a term of our sequence is zero, then
almost all the sequence is forced to be zero. The general case of
higher order  propagation was established by R. Curto and L. Fialkow
in \cite[Proposition 5.13]{CF}. Our  first contribution in this
section is to provide an elementary proof of this fact,  based on a
version of block matrices determinants.
 \begin{theorem}[Propagation phenomena for $k$-positive matrices]\label{main}
    Let $M\in \mathcal{C}^k_+$ be such that there exists an integer $ n_0\geq 0$ satisfying
    $\det([M_\gamma]^{n_0}_{k-1})=0$. Then $\det([M_\gamma]^{n}_{k-1})=0$, for all $n\geq 1$.
\end{theorem}
For weighted shifts this  propagation phenomena is formulated as
follows:
 \begin{theorem}[Curto-Fialkow \cite{CF}]\label{hypsub}
 Let $\alpha=\{\alpha_n\}_{n\geq 0}$ be a sequence of
positive numbers such that the weighted shift $W_\alpha$ is
$k$-hyponormal. We assume that $\det([M_\gamma]_{p}^{n_0})=0$ for
some $n_0\geq0$, $p<k$, then,

 $$
 \det([M_\gamma]_{p}^n)=0 \quad \textrm{ for all } n\ge 0.
$$
 In particular $W_\alpha$ is subnormal.
 \end{theorem}
Let $W_\alpha$ be a weighted shift, and let
$\gamma=\{\gamma_n\}_{n\geq 0}$ be its moment sequence. We will say
that $W_\alpha$  is {\it recursively generated } if there exist
$r\in \mathbb{N}^*$, $a_0,\cdots,a_{r-1}\in \mathbb{R}$ such that
for every $k\ge 0$,
   $$\gamma_{k+r}=a_{r-1}\gamma_{k+r-1}+\cdots +a_1\gamma_{k+1}+a_0\gamma_k.$$
Using Theorem \ref{hypsub}, we obtain  the following recursiveness
criterion for subnormal weighted shifts:
 \begin{theorem}[recursively generated subnormal weighted shift]\label{sub weigh
 shift}
Let $W_\alpha$ be a subnormal weighted shift and let
$\gamma=\{\gamma_n\}_{n\geq 0}$ be  its moment sequence. Then the
following conditions are equivalent:
\begin{itemize}
\item[(i)] $W_\alpha$ is recursively generated,
\item[(ii)] there exist $n_0,k\in \mathbb{N}$ such that
$\det([M_\gamma]_k^{n_0})=0$.
\end{itemize}
\end{theorem}
In connection with Berger's theorem, we get the next
characterization of positive measures with finite support:
\begin{theorem}[Dirac measure characterization]\label{dirac}
Let $\mu$ be a positive measure  with support in $\mathbb{R}^+$.
Then the following conditions are equivalent:
\begin{itemize}
\item[(i)] $\mu$ is a finite Dirac measure combination,
\item[(ii)] there exist  $p,k \in \mathbb{N}$ such that $\det(s_{i+j+p})_{i,j\leq k
  }=0$.
\end{itemize}
\end{theorem}
Let $W_\alpha$ be a $k$-hyponormal weighted shift. A rank one
perturbation of  $W_\alpha$  associated with $i\in \mathbb{N}$ and
$t>0$ is the weighted shift $W_{\alpha(l,t)}$  given by,
$\alpha(l,t)_{i}=\alpha_i$ if $i\neq l$ and
$\alpha(l,t)_l=t\alpha_l$. We also associate with $W_\alpha$, the
$k$-hyponormal stable set defined by
$$
I^k=\left\{t\geq 0 \; ;\; W_{\alpha(l,t)} \textrm{ is
$k$-hyponormal}\right\}.
$$
Clearly $1 \in I^k$. Furthermore,
\begin{theorem}[Rank-one perturbation]\label{I inetr}
$I^k$ is a nonempty bounded closed interval.
\end{theorem}
The last result concerns local perturbation, i.e. if there exist
$\epsilon > 0$ such that $]1-\epsilon,1+\epsilon[\in I^k$. More
precisely, we get:
\begin{theorem}[local perturbation]\label{I centrer}
Let $W_\alpha$ be a $k$-hyponormal weighted shift. Then,
$$
1 \in \overset{\circ}{\widehat{I^k}}\qquad \Longleftrightarrow
\qquad [M_\gamma]^n_k \textrm{ is positive definite } \forall n\leq
k.
$$
\end{theorem}
\section{Proof of Theorem \ref{main}}

We start by the introduction of some notations before expanding the
proof of Theorem \ref{main}.

For an $(n+1)\times(n+1)$-matrix $M=(a_{i,j})_{0\leq i,j\leq n}$,
and for $i_0,j_0\leq n$, we denote by $M(^{i_0}_{j_0})$ the $n\times
n-$matrix resulting by removing the $i_0+1$ row and $j_0+1$ column.
We also use the notation
 $M(^{0\;\; n}_{0\;\; n})$ for  the $(n-1)\times(n-1)$-matrix resulting by removing the first and the last row and column.
 In the case where  $j_0=0$ (resp $j_0=n$), we simply
 denote $M(^{i_0}_0)$ by
  $M(\widetilde{i_0})$ (resp $M(^{i_0}_n)$  by $M(\hat{i_0})$).

We  will use  the next expansion formula,
\begin{lemma}[Desnanot-Jacobi adjoint matrix]\label{desna}
If $M=(a_{i,j})_{0\leq i,j\leq n}$ is an  $(n+1)\times (n+1)$
matrix, then
\begin{equation}\label{desn}
\det(M)\det\left(M(^{0\;\; n}_{0\;\;
n})\right)=\det(M(\widetilde{0}))\det(M(\widehat{n}))-\det(M(\widehat{0}))\det(M(\widetilde{n})).
\end{equation}
\end{lemma}
\begin{proof}
For a proof of the previous lemma, we refer to \cite{Bre} Page 111.
\end{proof}
\begin{proof}[Proof of Theorem \ref{main}] Suppose that $\det([M_\gamma]_{k-1}^i])=0$, for
some $i\geq n_0$, and let us show that 
$\det([M_\gamma]_{k-1}^{j}])=0$ for every $j\in {\mathbb N}$. We will prove
that $\det([M_\gamma]_{k-1}^{i-1}])=\det([M_\gamma]_{k-1}^{i+1}])=0$.

For $M=[M_\gamma]_k^i$, we get $M(^{0\;\; k}_{0\;\;
k})=[M_\gamma]_{k-2}^{i+2}$,
$M(\widetilde{0})=[M_\gamma]_{k-1}^{i+2}$,
$M(\widehat{k})=[M_\gamma]_{k-1}^{i}$, and
$M(\widehat{0})=M(\widetilde{k})=[M_\gamma]_{k-1}^{i+1}$.
 Applying the identity (\ref{desn}), we get
\begin{equation}\label{Alice}
\det([M_\gamma]_k^i)\det\left([M_\gamma]_{k-2}^{i+2}\right)=\det([M_\gamma]_{k-1}^{i})\det([M_\gamma]_{k-1}^{i+2})-\det([M_\gamma]_{k-1}^{i+1})^2.
\end{equation}
So we get
$\det\left([M_\gamma]_{k-2}^{i+2}\right)\det([M_\gamma]_k^i)\ge 0$
and then:
\begin{equation}\label{32}
\det([M_\gamma]_{k-1}^{i})\det([M_\gamma]_{k-1}^{i+2})\ = -\det([M_\gamma]_{k-1}^{i+1})^2\le 0.
\end{equation}

Now  $M_\gamma\in\mathcal{C}^k_+$, implies  that $\det([M_\gamma]_{k-1}^{i+1})=0$.   Finally,
$\det([M_\gamma]_{k-1}^{n})=0$ for every $n\ge n_0$.

Replacing $i$, by $i-2$ in Equation \eqref{32}, we obtain 

\begin{equation}\label{32}
\det([M_\gamma]_{k-1}^{i-2})\det([M_\gamma]_{k-1}^{i})\ = -\det([M_\gamma]_{k-1}^{i-1})^2\le 0.
\end{equation}
Which proves that
$\det([M_\gamma]_{k-1}^{i})=0$ and then that
$\det([M_\gamma]_{k-1}^{j})=0$, for every $j\in {\mathbb N}$. This
completes the proof.
\end{proof}

\section{Propagation phenomena for $k$-hyponormal weighted shifts}
We recall the Stampfli propagation result for subnormal weighted
shift \cite[Theorem 6]{Sta} (see also \cite[Proposition 4.5]{BRZ}):
\begin{theorem}
Let $W_\alpha$ be a injective hyponormal weighted shift, if we
assume that $W_\alpha$ is subnormal and that
$\alpha_{i_0}=\alpha_{i_0+1}$ for some arbitrary $\alpha_0$. Then,
$$
\alpha_i=\alpha_{i+1},\;\textrm{ for all } \; i\geq 1.
$$
\end{theorem}
In Term of moment sequence associated to $W_\alpha$, we can write
the result as :
\begin{theorem}
Let $W_\alpha$ be a subnormal weighted shift, and let $\gamma$ the
associated moment, if we assume that for some $i_0\geq 0$, we have
$\left|\begin{array}{cc}
                                                  \gamma_{i_0} & \gamma_{i_0+1} \\
                                                  \gamma_{i_0+1} & \gamma_{i_0+2}
                                                \end{array}
 \right|=0$. Then,
$$
\left|\begin{array}{cc}
                                                  \gamma_{i} & \gamma_{i+1} \\
                                                  \gamma_{i+1} & \gamma_{i+2}
                                                \end{array}
 \right|=0,\;\textrm{ for all } \; i\geq 1.
$$
\end{theorem}
The first extension of the propagation notion was given by Curto in
\cite{Cur90}, in fact he prove that:
\begin{theorem}[{\cite[Corollary 6]{Cur90}}]
Let $W_\alpha$ be a 2-hyponormal weighted shift, if we assume that
$\alpha_{i_0}=\alpha_{i_0+1}$. Then,
$$
\alpha_i=\alpha_{i+1},\;\textrm{ for all } \; i\geq 1, \textrm{ and
} \alpha_0, \textrm{ is arbitrary}.
$$
\end{theorem}
 In the light of this results, we will extend the notion of
propagation phenomena for $k$-hyponormal weighted shifts. Recall
that if $W_\alpha$ is a weighted shift with bounded weight sequence
$\alpha=\{\alpha_n\}_{n\geq0}$, the moments of $W_\alpha$ are
usually defined by $\gamma_0:=1$ and
$\gamma_{n+1}:=\alpha_n^2\gamma_n$ ($n\geq 0$). It is known that
(see \cite[Theorem 4]{Cur90}) $W_\alpha$ is $k$-hyponormal if and
only if $[M_\gamma]_k^n \geq 0$ for all $n\geq 0$.

We retrieve the next theorem du to Curto and Fialkow:

 \begin{theorem}[{\cite[Proposition 5.13]{CF}}]\label{hyp sub}
 Let $\alpha=\{\alpha_n\}_{n\geq 0}$ be a sequence of
positive numbers such that the weighted shift $W_\alpha$ is
$k$-hyponormal. Assume that $\det([M_\gamma]_{p}^{n_0})=0$ for
some $n_0\geq0$ and $k<p$, then,
 $$
 \det([M_\gamma]_{p}^n)=0 \quad \textrm{ for all } n\ge 0.
$$
 In particular $W_\alpha$ is subnormal.
 \end{theorem}
\begin{proof}
Since $W_\alpha$ is $k$-hyponormal if and only if $M_\gamma$ is
$k$-positive, Then $p+1$-positive. Since also
$\det([M_\gamma]_{p}^{n_0})=0$, Theorem \ref{main} implies that
$\det([M_\gamma]_{p}^n)=0$ for all $n\geq 1$. we deduce that
$\det([M_\gamma]_{l}^n)=0$ for all $l\geq k$ and $n\geq 0$. It
follows that $[M_\gamma]_l^n$ is positive semi-definite for all
$l\geq k$ and $n\geq 0$, and in particular that $W_\alpha$ is
subnormal.
\end{proof}
 For $k=2$, the next extension of Stampfli's propagation result \cite[Proposition 4.5]{BRZ}, can be
 found in \cite{Cur90}.
 \begin{corollary} \cite[Corollary 6]{Cur90} \label{curto prop}
  Let $\alpha=\{\alpha_n\}_{n\ge 0}$ be a bounded sequence of
positive numbers associated with a $2$-hyponormal weighted shift. If
 $\alpha_{n_0}=\alpha_{n_0+1}$ for some $n_0\geq0$, then
$\alpha_n=\alpha_{n_0}$ for all $n\geq1$. In particular, $W_\alpha$
is subnormal.
 \end{corollary}

\section{Moment characterization of finite mass measures}
We recall that a measure $\mu$ is said to be of finite mass point if $\mu$ has finite support or equivalently, 
$\mu$ is a finite combination of Dirac measures.\\
The question of characterizing finite mass measure in terms of
moments, has been intensively studied and arises in several branches where finite interpolation is needed. It is closely related to truncated moment problem.  For charges
(non necessary positive measures ), we can see  Corollary 4.3 in 
\cite{CRZ}. In the case of positive measures on $\mathbb{R}$, a more recent work of C. Berg and D. Szwarc provides an interesting study, 
see Theorem 1.1 of \cite{CB}. In this section we prove  Theorem   \ref{sub weigh shift} that gives a simple
characterization of finite mass measure on $\mathbb{R}^+$.
\begin{proof}[Proof of Theorem \ref{sub weigh shift}]\qquad

\begin{description}
\item[$(i)\Rightarrow (ii)$] If that $W_\alpha$ is recursively generated, there exists  $k\geq 0$ and
$a_0,a_1,\dots,a_{k}\in \mathbb{R}$ such that
$$
\gamma_{p+k+1}=\sum_{j=0}^{k} a_j \gamma_{p+j}\qquad \mbox{ for every }  p\in
\mathbb{N}.
$$
It follows, $\det([M_\gamma]_{k+1}^p)=0$.

\item[$(ii)\Rightarrow (i)$] From  $W_\alpha$ is subnormal, we derive in particular  that $W_\alpha$ is
$k+1$-hyponormal. Now  from the condition $\det([M_\gamma]_k^{n_0})=0$
together with Theorem \ref{hyp sub}, we obtain 
$$\det([M_\gamma]_{k+1}^n)=0,\; \; \mbox{ for every }  n\in
\mathbb{N}.$$  We deduce  that there exists
$a_0,a_1,\dots,a_{k}\in \mathbb{R}$ such that
$$
\gamma_{p+k+1}=\sum_{j=0}^{k} a_j \gamma_{p+j}\qquad \mbox{ for every }  p\in
\mathbb{N}.
$$
Finally, $W_\alpha$ is recursively generated.
\end{description}
\end{proof}
We use the previous theorem to exhibit a characterization of
measures with finitely mass points in  $\mathbb{R}_+$:
\begin{theorem}[Theorem \ref{dirac}]
Let $\mu$ be a positive measure. We suppose that $\mu$ is  supported
in $\mathbb{R}_+$ and has  finite moments $s_k$ for every $k\geq 0$.
Then, $\mu$ is a finite Dirac measure combination if and only if
there exist $p,k \in \mathbb{N}$ such that
$\det(s_{i+j+p})_{i,j\leq k}=0$.
\end{theorem}
\begin{proof}[Proof of Theorem \ref{dirac}]
Clearly, if $\displaystyle\mu=\sum_{l=0}^n a_l
  \delta_{x_l}$, then from \cite[Lemma 2.1]{CB}, there exist $p,k \in \mathbb{N}$ such that  $\det(s_{i+j+p})_{i,j\leq k
  }=0$.
For the reverse implication,  by Remark \eqref{rq} - \eqref{3},  the
matrix $(s_{i+j})_{i,j}$ is $k$-positive for all $k\in \mathbb{N}$.
The weighted shift $T_\alpha$ associated with the sequence
$\alpha_n = \sqrt{\frac{s_{n+1}}{s_n}}$ is subnormal. It is also
recursively generated by Theorem \ref{sub weigh shift}, and by
appealing Theorems 3.5 and 3.6  in \cite{Cur}, $\mu$ has a finite
mass.
\end{proof}
\section{Perturbation of k-positive matrices}
Let $M\in \mathcal{C}_+^k(H)$ ($M=M_\gamma$). For $l\in\mathbb{N}^*$
and $t\geq 0$ we denote by  $M_{\gamma'}$ the perturbed Hankel
matrix whose entries are given by
 $$
\gamma'_n= \left\{\begin{array}{ll}
 \gamma_n, &\: \:  if\: \:  n\le l;\\
t\gamma_n, &   \: \:  if\: \:    n \ge l+1.
\end{array} \right. $$
Where
$$
M_{\gamma'}=\left(
\begin{array}{ccccccc}
  \gamma_0 & \gamma_1 & \dots&\gamma_l & t\gamma_{l+1} & t\gamma_{l+2}& \dots \\
   \gamma_1 & \dots & \gamma_l & t\gamma_{l+1} & t\gamma_{l+2}& \dots & \dots\\
   \vdots & \dots & t\gamma_{l+1} & t\gamma_{l+2} & \dots & \dots &\dots\\
   \gamma_l & t\gamma_{l+1} & t\gamma_{l+2} & \dots & \dots & \dots&\dots\\
   t\gamma_{l+1} & t\gamma_{l+2} & \dots & \dots & \dots&   \dots& \dots\\
   t\gamma_{l+2} & \dots & \dots & \dots & \dots&   \dots& \dots\\
   \vdots & \dots & \dots &\dots & \dots&   \dots& \dots\\
\end{array}
\right).
$$
The main goal of this section is to determine when such perturbation
of $k$-positive matrix remains $k$-positive.

We notice that for $n\geq l+1$, we get
$[M_{\gamma'}]^n_k=t[M_\gamma]^n_k$ and we deduce that
\begin{equation}\label{1'}
M_{\gamma'} \textrm{ is k-positive }\Longleftrightarrow
[M_{\gamma'}]^n_k \textrm{ is positive semi-definite }\forall n\leq
l.
\end{equation}

For $n\leq l$, we write
\begin{equation}\label{Convex}
[M_{\gamma'}]^n_k=t [M_\gamma]^n_k+(1-t) H_k^n(l),
\end{equation}
with
$$
H_k^n(l)=\left(
              \begin{array}{ccccccc}
                \gamma_{n} & \gamma_{n+1} & \cdots & \gamma_{l}& 0&\cdots &0 \\
                \gamma_{n+1} &  \gamma_{n+2} & \cdots &  0 & 0&\cdots &0 \\
                \vdots & \vdots & \ddots & \vdots & &\cdots &0 \\
                \gamma_{l} & 0 & \ddots & \vdots & &\cdots &0 \\
               0 & \vdots & \ddots & \vdots &&\cdots &0 \\
               \vdots & \vdots & \ddots & \vdots & &\cdots &0 \\
               0 & 0 & 0 & 0 & 0&0 &0 \\
              \end{array}
            \right),
$$
and we denote $ I_n^k=\left\{t\geq 0 / \quad t [M_\gamma]^n_k+(1-t)
H_k^n(l) \textrm{ is positive semi definite} \right\}.
$\\
 We have the following property,
\begin{proposition}\label{2}
 For every $n\leq l$,  $I_n^k$ is a non empty closed interval of $\mathbb{R}_+$.
\end{proposition}
\begin{proof}For $k\in \mathbb{N}$, let $P^+_k$ be the closed convex cone of $(k+1)\times (k+1)$
positive semi definite matrices, and denote by  $D=\left\{t
[M_\gamma]^n_k+(1-t) H_k^n(l) / \quad t\geq 0 \right\}$. It's
obvious that $D$ is a closed convex set of $\mathcal{M}_{k+1}$ and
then $D\cap P^+_k$ is a closed convex set.
Moreover, the mapping $\gamma_n^k$, given by
$$
\begin{array}{cccl}
  \gamma_n^k : & \mathbb{R}_+ & \longrightarrow & D \\
   & t &\longmapsto & t [M_\gamma]^n_k+(1-t) H_k^n(l),
\end{array}
$$
is a continuous affine function, and hence $I_k^n=
(\gamma_n^k)^{-1}(D\cap P^+_k)$ is a closed interval of
$\mathbb{R}_+$. Also, $M$  is $k$-positive implies that
$[M_\gamma]_k^n$ is positive for all $n\in \mathbb{N}$, and hence
that $1 \in I_n^k$ for all $n\in \mathbb{N}$.
\end{proof}
  From \eqref{1'}, a perturbation of
$k$-positive matrix remains $k$-positive if and only if $t\in
I^k:=\cap_{n\leq l} I_n^k$,  our  second result  is
\begin{proposition}\label{2.}
For every $k\geq 1$, $I^k$ is a compact interval.
\end{proposition}
\begin{proof} From proposition \ref{2} we have $I^k$ is a closed
 non empty interval.
To  show  that $I^k$ is bounded, we remark that if
$[M_{\gamma'}]_k^n$ is positive semi-definite, we have
 $[M_{\gamma'}]^n_{k-1}$ is positive semi
definite. We deduce that $I_{n}^k\subset I_n^{k-1}$, and by
induction that $I_n^k\subset I_n^1$. Finally $I^k\subset I^1$.

Now $t\in I^1$ if and only if ${M_{\gamma'}}_{l-1}^1$ and
${M_{\gamma'}}_{l}^1$ have non-negative determinants, that is
$$
\frac{\gamma_l^2}{\gamma_{l-1}\gamma_{l+1}}\leq t\leq \frac{\gamma_l
\gamma_{l+2}}{\gamma_{l+1}^2}.
$$
Thus for every $k \ge 1$, $I^k \subset
I^1=\left[\frac{\gamma_l^2}{\gamma_{l-1}\gamma_{l+1}};
\frac{\gamma_l \gamma_{l+2}}{\gamma_{l+1}^2} \right]$ is bounded.
\end{proof}
\subsection{Determination of $I^2$} The problem of determining $I^k$ for $k\ge 2$ seems to be complicated. From the previous proof, we see that  $I^1=\left[\frac{\gamma_l^2}{\gamma_{l-1}\gamma_{l+1}};
\frac{\gamma_l \gamma_{l+2}}{\gamma_{l+1}^2} \right]$. We devote
this section to  calculate $I^2$.

Since  $t\in I^2$  if and only if $[M_{\gamma'}]_{l-3}^2$ and
$[M_{\gamma'}]_{l-2}^2$ and $[M_{\gamma'}]_{l-1}^2$ and
$[M_{\gamma'}]_{l}^2$ are positives semi-definite,we exhibit the
corresponding condition in each case :
\begin{description}
  \item[$\bullet    \: {[M_{\gamma'}]}_{l-3}^2\ge 0$].   We recall that
  $$
  [M_{\gamma'}]_{l-3}^2=\left(
                 \begin{array}{ccc}
                   \gamma_{l-3} & \gamma_{l-2} & \gamma_{l-1} \\
                   \gamma_{l-2} & \gamma_{l-1} & \gamma_l \\
                   \gamma_{l-1} & \gamma_{l} & t\gamma_{l+1} \\
                 \end{array}
               \right),
  $$
  and then is  positive semi-definite if and only if:
  $$
    \left|
                 \begin{array}{cc}
                    \gamma_{l-1} & \gamma_l \\
                    \gamma_{l} & t\gamma_{l+1} \\
                 \end{array}
               \right|\geq 0 ,
    \left|
                 \begin{array}{ccc}
                   \gamma_{l-3} &  \gamma_{l-1} \\
                   \gamma_{l-1} &  t\gamma_{l+1} \\
                 \end{array}
               \right|\geq 0  \: \: \textrm{ and } \: \:
    \left|
                 \begin{array}{ccc}
                   \gamma_{l-3} & \gamma_{l-2} & \gamma_{l-1} \\
                   \gamma_{l-2} & \gamma_{l-1} & \gamma_l \\
                   \gamma_{l-1} & \gamma_{l} & t\gamma_{l+1} \\
                 \end{array}
               \right| \geq 0.
$$
Which is equivalent to
$$\left\{  \begin{array}{lcl}
    \/t &\geq & max( \frac{\gamma_l^2}{\gamma_{l-1}\gamma_{l+1}}, \frac{\gamma_{l-1}^2}{\gamma_{l-3}\gamma_{l+1}}),\\
\:  t  &\geq& \frac{\left|
                 \begin{array}{ccc}
                   \gamma_{l-3} & \gamma_{l-2} & \gamma_{l-1} \\
                   \gamma_{l-2} & \gamma_{l-1} & \gamma_l \\
                   \gamma_{l-1} & \gamma_{l} & 0\\
                 \end{array}
               \right|}{\gamma_{l+1} \left|
                 \begin{array}{ccc}
                   \gamma_{l-3} & \gamma_{l-2} \\
                   \gamma_{l-2} & \gamma_{l-1} \\
                 \end{array}
               \right|}.
  \end{array}\right.
$$
 The last condition is redundant. Indeed, using  Lemma \ref{desna} (also called, Dodgson condensation
  method), we obtain :
  $$
\left|
                 \begin{array}{ccc}
                   \gamma_{l-3} & \gamma_{l-2} & \gamma_{l-1} \\
                   \gamma_{l-2} & \gamma_{l-1} & \gamma_l \\
                   \gamma_{l-1} & \gamma_{l} & 0\\
                 \end{array}
               \right|=\frac{-1}{\gamma_{l-1}}\left[ \gamma_l^2\left|
                 \begin{array}{cc}
                   \gamma_{l-3} & \gamma_{l-2}\\
                   \gamma_{l-2} & \gamma_{l-1}\\
                 \end{array}
               \right|+\left|
                 \begin{array}{cc}
                   \gamma_{l-2} & \gamma_{l-1}\\
                   \gamma_{l-1} & \gamma_{l}\\
                 \end{array}
               \right|^2\right]\leq 0.
  $$
By  using this simple next observation:
$$
\frac{\gamma_l^2}{\gamma_{l-1}\gamma_{l+1}}=\frac{\gamma_{l-1}^2}{\gamma_{l-3}\gamma_{l+1}}\frac{\gamma_{l-3}\gamma_{l-1}}{\gamma_{l-2}^2}\left(\frac{\gamma_{l-2}\gamma_{l}}{\gamma_{l-1}^2}\right)^2\geq\frac{\gamma_{l-1}^2}{\gamma_{l-3}\gamma_{l+1}}.
$$
It follows that
\begin{equation}\label{matrix1}
[M_{\gamma'}]_{l-3}^2 \textrm{ is positive semidefinite if and only
if } t\geq \frac{\gamma_l^2}{\gamma_{l-1}\gamma_{l+1}}.
\end{equation}

  \item[$\bullet   \: {[M_{\gamma'}]}_{l-2}^2\ge 0$] Using the same argument for
  $$
  [M_{\gamma'}]_{l-2}^2=\left(
                 \begin{array}{ccc}
                    \gamma_{l-2} & \gamma_{l-1} &\gamma_{l} \\
                    \gamma_{l-1} & \gamma_l &t\gamma_{l+1} \\
                    \gamma_{l} & t\gamma_{l+1} &t\gamma_{l+2} \\
                 \end{array}
               \right).
  $$
We obtain,
  $$
    \left|
                 \begin{array}{cc}
                    \gamma_{l-2} & \gamma_l \\
                    \gamma_{l} & t\gamma_{l+2} \\
                 \end{array}
               \right| \geq 0 ,
    \left|
                 \begin{array}{ccc}
                   \gamma_{l} &  \gamma_{l+1} \\
                  t \gamma_{l+1} & \gamma_{l+2} \\
                 \end{array}
               \right| \geq  0 \mbox{ and }
    \left|\begin{array}{ccc}
                    \gamma_{l-2} & \gamma_{l-1} &\gamma_{l} \\
                    \gamma_{l-1} & \gamma_l &t\gamma_{l+1} \\
                    \gamma_{l} & t\gamma_{l+1} &t\gamma_{l+2} \\
                 \end{array}
               \right|\geq  0.
$$
That implies, $ \frac{\gamma_l^2}{\gamma_{l-2}\gamma_{l+2}}\le t \le
\frac{\gamma_{l}\gamma_{l+2}}{\gamma_{l+1}^2}$ and that the quadric polynomial, 
 $$
P(t):=-\gamma_{l-2}\gamma_{l+1}^2 t^2 +
               (\gamma_{l-2}\gamma_l
               \gamma_{l+2}+2\gamma_{l-1}\gamma_l\gamma_{l+1}-\gamma_{l-1}^2\gamma_{l+2})t-\gamma_l^3,
 $$
is positive.  The first and the second inequalities are automatically satisfied,
since
  $$
P\left(\frac{\gamma_l^2}{\gamma_{l-2}\gamma_{l+2}}\right)=-\frac{(\gamma_l^2\gamma_{l+1}-\gamma_{l-1}\gamma_l\gamma_{l+2})^2}{\gamma_{l+2}^2\gamma_{l-2}}\leq0,
$$
and
$$
P\left(\frac{\gamma_l\gamma_{l+2}}{\gamma_{l+1}^2}\right)=-\frac{\gamma_l(\gamma_{l}\gamma_{l+1}-\gamma_{l-1}\gamma_{l+2})^2}{\gamma_{l+1}^2}\leq0.
$$
Using the inequality $\gamma_{l-1}^2 \leq \gamma_{l-2} \gamma_{l}$, 
we  drive that the second coefficient of $P$ is positive, and
hence by the classical Descartes rule for positive roots of
polynomials,  we derive  that $P$ has two distinct  positive roots.
  Then,
\begin{equation}\label{matrix2}
[M_{\gamma'}]_{l-2}^2 \textrm{ is positive semidefinite if and only
if } t\in\left[\alpha(P) ; \beta(P) \right],
\end{equation}
where $\alpha(P)$ and $\beta(P)$ are the two positive solutions of 
$P(t)=0$.
  \item[$\bullet \: \: {[M_{\gamma'}]}_{l-1}^2$] This case is  treated exactly as the last one
  and leads to
\begin{equation}\label{matrix3}
[M_{\gamma'}]_{l-1}^2 \textrm{ is positive semidefinite if and only
if } t\in\left[\alpha(Q) ; \beta(Q) \right],
\end{equation}
where $\alpha(Q)$ and $\beta(Q)$ are the two positive solutions of
$$
Q(t):=-\gamma_{l+1}^3
t^2+(\gamma_{l-1}\gamma_{l+1}\gamma_{l+3}+2\gamma_l\gamma_{l+1}\gamma_{l+2}-\gamma_{l-1}\gamma_{l+2}^2)t-\gamma_l^2\gamma_{l+3}
=0.
$$
  \item[$\bullet \: {[M_{\gamma'}]}_{l}^2$] The computations in this  case  outlines  the first one,
  indeed,
$$
  [M_{\gamma'}]_{l}^2=\left(
                 \begin{array}{ccc}
                   \gamma_{l} & t\gamma_{l+1} & t\gamma_{l+2} \\
                   t\gamma_{l+1} & t\gamma_{l+2} & t\gamma_{l+3} \\
                   t\gamma_{l+2} & t\gamma_{l+3} & t\gamma_{l+4} \\
                 \end{array}
               \right),
  $$
  is  positive semi-definite if and only if
  $$
    \left|
                 \begin{array}{cc}
                    \gamma_{l} & \gamma_{l+1} \\
                    t\gamma_{l+1} & \gamma_{l+2} \\
                 \end{array}
               \right|\geq 0,
    \left|
                 \begin{array}{ccc}
                   \gamma_{l} &  \gamma_{l+2} \\
                   t\gamma_{l+2} &  \gamma_{l+4} \\
                 \end{array} \right|\geq 0
              \mbox{ and}
    \left|
                 \begin{array}{ccc}
                   \gamma_{l} & \gamma_{l+1} & \gamma_{l+2} \\
                   t\gamma_{l+1} & \gamma_{l+2} & \gamma_{l+3} \\
                   t\gamma_{l+2} & \gamma_{l+3} & \gamma_{l+4} \\
                 \end{array}
               \right| \geq 0.
               $$
Which is equivalent to
$$  t \leq min\left\{ \frac{\gamma_{l}\gamma_{l+2}}{\gamma_{l+1}^2},   \frac{\gamma_{l}\gamma_{l+4}}{\gamma_{l+2}^2},  \frac{-\gamma_{l} \left|
                 \begin{array}{ccc}
                   \gamma_{l+2} & \gamma_{l+3} \\
                   \gamma_{l+3} & \gamma_{l+4} \\
                 \end{array}
               \right|}{\left|
                 \begin{array}{ccc}
                   0 & \gamma_{l+1} & \gamma_{l+2} \\
                   \gamma_{l+1} & \gamma_{l+2} & \gamma_{l+3} \\
                   \gamma_{l+2} & \gamma_{l+3} & \gamma_{l+4} \\
                 \end{array}
               \right| }\right\}.$$
But
$$
  \frac{\gamma_l\gamma_{l+2}}{\gamma_{l+1}^2} = \frac{\gamma_l\gamma_{l+4}}{\gamma_{l+2}^2}\frac{\gamma_{l+3}^2}{\gamma_{l+4}\gamma_{l+2}}\left(\frac{\gamma_{l+2}^2}{\gamma_{l+3}\gamma_{l+1}}\right)^2 \leq
\frac{\gamma_l\gamma_{l+4}}{\gamma_{l+2}^2},
$$
  with
  $$
\frac{\gamma_l\gamma_{l+2}}{\gamma_{l+1}^2}+ \frac{\gamma_{l} \left|
                 \begin{array}{ccc}
                   \gamma_{l+2} & \gamma_{l+3} \\
                   \gamma_{l+3} & \gamma_{l+4} \\
                 \end{array}
               \right|}{\left|
                 \begin{array}{ccc}
                   0 & \gamma_{l+1} & \gamma_{l+2} \\
                   \gamma_{l+1} & \gamma_{l+2} & \gamma_{l+3} \\
                   \gamma_{l+2} & \gamma_{l+3} & \gamma_{l+4} \\
                 \end{array}
               \right|}  = \frac{\gamma_l\gamma_{l+2}}{\gamma_{l+1}^2}\left[
               1- \frac{\left|
                 \begin{array}{ccc}
                   \gamma_{l+2} & \gamma_{l+3} \\
                   \gamma_{l+3} & \gamma_{l+4} \\
                 \end{array}
               \right|}{\left|
                 \begin{array}{ccc}
                   \gamma_{l+2} & \gamma_{l+3} \\
                   \gamma_{l+3} & \gamma_{l+4} \\
                 \end{array}
               \right|+\frac{1}{\gamma_{l+1}^2}\left|
                 \begin{array}{ccc}
                   \gamma_{l+1} & \gamma_{l+2} \\
                   \gamma_{l+2} & \gamma_{l+3} \\
                 \end{array}
               \right|^2}\right] \geq  0.
  $$
This yields
\begin{equation}\label{matrix4}
[M_{\gamma'}]_{l}^2 \textrm{ is positive semidefinite if and only if
}0 \leq t \leq \frac{-\gamma_{l} \left|
                 \begin{array}{ccc}
                   \gamma_{l+2} & \gamma_{l+3} \\
                   \gamma_{l+3} & \gamma_{l+4} \\
                 \end{array}
               \right|}{\left|
                 \begin{array}{ccc}
                   0 & \gamma_{l+1} & \gamma_{l+2} \\
                   \gamma_{l+1} & \gamma_{l+2} & \gamma_{l+3} \\
                   \gamma_{l+2} & \gamma_{l+3} & \gamma_{l+4} \\
                 \end{array}
               \right|}.
\end{equation}
\end{description}
Finally from \eqref{matrix1}, \eqref{matrix2}, \eqref{matrix3} and
\eqref{matrix4} we conclude that,
$$
I^2=\left[\max\left\{\alpha(P),\alpha(Q),\frac{\gamma_l^2}{\gamma_{l+1}\gamma_{l-1}}\right\};
\min\left\{\beta(P),\beta(Q), \frac{-\gamma_{l} \left|
                 \begin{array}{ccc}
                   \gamma_{l+2} & \gamma_{l+3} \\
                   \gamma_{l+3} & \gamma_{l+4} \\
                 \end{array}
               \right|}{\left|
                 \begin{array}{ccc}
                   0 & \gamma_{l+1} & \gamma_{l+2} \\
                   \gamma_{l+1} & \gamma_{l+2} & \gamma_{l+3} \\
                   \gamma_{l+2} & \gamma_{l+3} & \gamma_{l+4} \\
                 \end{array}
               \right|}\right\} \right].
$$
\begin{remarks} \qquad

$\bullet $ The existence of positive roots for $P$ and $Q$ can be
derived also by using the classical Bolzano's theorem because the
fact that $P(1)\geq 0$, $Q(1)\geq 0$, and $P(0)\leq 0$, $Q(0)\leq
0$, and that $\lim_{t\to \infty} P(t)=\lim_{t\to \infty}
Q(t)=-\infty$.

$ \bullet $ A  direct calculation of discriminants will give  the
next inequality,
$$
\min\{\gamma_{l-1}^2\gamma_{l+1}\gamma_{l+3};\gamma_{l-2}\gamma_l\gamma_{l+2}^2\}\geq
(2\gamma_l\gamma_{l+1}-\gamma_{l-1}\gamma_{l+2})^2.
$$
\end{remarks}
\subsection{One rank perturbation of weighted shifts} Let  $l\in\mathbb{N}$  be given,  $t\geq $ and let $W_\alpha$ be a weighted shift. The $(l,t)$-one perturbation of $W_\alpha$   is the weighted shift $W_{\alpha(l,t)}$  
 defined by: $W_{\alpha(l,t)}(e_k)=W_\alpha(e_k)=\alpha_k e_{k+1}$
for $k\neq l$ and $W_{\alpha(l,t)}(e_l)=\sqrt{t}\alpha_l e_{l+1}$.

We denote by $\gamma$ the moment  sequence associated with  $\alpha$
defined by $\gamma_0=1$ and $\gamma_n=\alpha_{n-1}^2\gamma_{n-1}$
and by $\gamma'(t)$ (Or simply $\gamma'$) the moment sequence
associated with $\alpha(l,t)$. It is   easily seen that
$\gamma'_k=\gamma_k$ for $k\leq l$ and $\gamma'_k=t\gamma_k$ for
$k>l$.

We put $J^k=\left\{t\geq 0 ; W_{\alpha(l,t)} \textrm{ is
$k$-hyponrmal } \right\}$. Since $W_\alpha$ is $k$-hyponormal if and
only if $M_\gamma$ is $k$-positive. we can deduce that,
$$
J^k=\left\{t\geq 0 ; W_{\alpha(l,t)} \textrm{ is $k$-hyponrmal}
\right\}=\left\{t\geq 0 ; M_{\gamma'} \textrm{ is $k$-positive}
\right\}=I^k.
$$

 and then from Proposition \ref{2.}, we get $I^k$ is a
compact interval.

The behavior of one rank  perturbation of subnormal weighted shift
is given by

\begin{theorem}[{\cite[Theorem 2.1]{Cur2}}]
Let $W_\alpha$ be a subnormal weighted shift. We have
$$I^\infty:=\cap_{{k\geq 1}}I^k=\{1\}.$$
\end{theorem}
We state our main result in this section as follows.
\begin{theorem}
Let $W_\alpha$ be a $k$-hyponormal weighted shift. We have
$$
1 \in \overset{\circ}{\widehat{I^k}}\qquad \Longleftrightarrow
\qquad [M_\gamma]^n_k \textrm{ is positive definite } \forall n\leq
l.
$$
\end{theorem}
We will use the same notations as in the proof of Proposition
\ref{2}. We start by  proving  that our condition is sufficient. To
this aim, consider
$$
\begin{array}{cccl}
  \gamma_n^k : & \mathbb{R}_+ & \longrightarrow & D \\
   & t &\longmapsto & t [M_\gamma]^n_k+(1-t) H_k^n(l)
\end{array}.
$$
Since $[M_\gamma]^n_k\in \overset{\circ}{\widehat{P^+_k}}$( the set
of all positive definite  matrices) and  $\gamma_n^k$ is continuous,
for $V$ an  open neighborhood of $[M_\gamma]_k^n$ such that
$V\subset \overset{\circ}{\widehat{P^+_k}}$, we get 
$(\gamma_n^k)^{-1}(V\cap D)$ is an open neighborhood of
$(\gamma_n^k)^{-1}([M_\gamma]_k^n)=1$. We deduce that there exists
$]r_n,t_n[\subset \mathbb{R}^+ $, such that $1\in ]r_n,t_n[$ and
$\gamma_n^k(]r_n,t_n[)\subset P^+_k$, and then $]r_n,t_n[\subset
I_n^k$. Finally
$$ 1\in ]\max_{n\leq l} r_n , \inf_{n\leq l} t_n[ \subset
\overset{\circ}{\widehat{I^k}}.$$

Conversely, we assume that there is  $n\leq l$ such that
$[M_\gamma]_k^n$ is not positive definite, or equivalently, there is
$p\leq k$ such that $\det([M_\gamma]^n_p)=0$. We distinguish two
cases;
\begin{description}
  \item [$p<k$] By Theorem \ref{hyp sub} $W_\alpha$ is subnormal,  then $I^k=I^\infty=\{1\}$ that means  $\overset{\circ}{\widehat{I^k}}=\emptyset$.
  \item [$p=k$] Simple computations  give
  $$
\det(\gamma_k^n(t))=t^{k+1}\det([M_\gamma]_k^n)+(1-t)t^k \gamma_n
Cof(\gamma_n)=\gamma_n Cof(\gamma_n) (1-t)t^k, 
  $$
where $Cof(\gamma_n)$ stands for the $(1,1)$ cofactor of
$[M_\gamma]_k^n$.  It follows from $t\in I^k$ that
$\det(\gamma_k^n(t))\geq 0$  and hence $t\leq 1$. Thus
$1\not\in\overset{\circ}{\widehat{I^k}}$.
\end{description}

\bigskip
\textbf{Acknowledgments.} The first  and the last  author are supported by the Project URAC 03
of the National center of research and by  the Hassan II academy of
sciences.  The first author is partially supported by the CNRST
grant $24UM5R2015$, Morocco.

\bibliographystyle{amsplain}

\end{document}